%% file: main.tex
\documentclass{siamart171218}

\input{ex_shared}
\ifpdf
\hypersetup{
pdftitle={\TheTitle},
pdfauthor={\TheAuthors}
}
\fi

\begin{document}
\maketitle

\begin{abstract}
We analyse the forward error in the floating point summation of real numbers, 
from algorithms that do not require
recourse to higher precision or better hardware. We derive 
informative explicit expressions,
and new deterministic and probabilistic bounds for errors in three classes
of algorithms: general summation,
shifted general summation, and compensated (sequential) summation.
Our probabilistic bounds for general and shifted general summation hold to
all orders. For compensated summation, we also present deterministic and
probabilistic first and second order bounds, with a first order
bound that differs from existing ones. 
Numerical experiments illustrate that the bounds are informative and that  
among the three algorithm classes, compensated summation is generally the most accurate 
method. 
\end{abstract}



\section{Introduction}
Given $n$ real numbers $x_1, \ldots, x_n$, we consider algorithms for 
summation $s_n=x_1+\cdots +x_n$ in floating point arithmetic. The algorithms operate
in one and only one precision, without any recourse to higher precision
as in \cite{blanchard2020class}, or wider accumulators as in \cite{DemmelH03}.

We analyze the forward
error $e_n=\hat{s}_n-s_n$ in the computed sum $\hat{s}_n$ in terms of the unit roundoff $u$,
for three classes of algorithms: general summation (Section~\ref{s_general}), 
shifted general summation (Section~\ref{s_center}), 
and compensated summation (Section~\ref{s_compensated}). Numerical experiments
(Section~\ref{s_numex}) illustrate that the bounds are informative, and that,
among the three classes of algorithms, compensated summation is the most accurate
method. Hybrid algorithms as in \cite{blanchard2020class} will be
analyzed in future work.

\subsection{Contributions}
We present informative bounds, based on clear and structured derivations, 
for three classes of summation algorithms.

\paragraph{General summation}
We derive explicit expressions for the error
(Lemmas \ref{lemma:forwardErrorGenl} and~\ref{lemma:forwardErrorRec});
as well as deterministic bounds  that hold to
all orders (Theorem~\ref{t_gdet})
and depends on the height $h$ of the computational tree associated with the particular
summation order.
We present two probabilistic bounds that treat the roundoffs as random variables,
one for independent roundoffs
(Theorem \ref{thm:model1Theorem}) and one for mean-independent roundoffs 
(Theorem~\ref{thm:model2Theorem}). 
Both bounds hold to all orders, depend only on $\sqrt{h}$, and
suggest that reducing the tree height~$h$ increases the accuracy. 

\paragraph{Shifted general summation}
We present the extension of shifted \textit{sequential} summation to shifted 
\textit{general} summation (Algorithm \ref{alg:shiftSum}).
For the special case of shifted \textit{sequential} summation we present an explicit
expression for the error, and a simple probabilistic bound for independent roundoffs
(Theorem~\ref{t_sprob}). For shifted \textit{general} summation, we deduce
probabilistic bounds for mean-independent roundoffs (Theorem~\ref{t_sprob2})
and independent roundoffs (Theorem~\ref{t_sprob3}) as respective
special cases of Theorems \ref{thm:model2Theorem} and~\ref{thm:model1Theorem}.

\paragraph{Compensated summation}
We derive three different explicit expressions for the error 
(Theorems \ref{thm:compExpr1} and~\ref{t_cs4}),
and bounds that are valid to first and second order, both deterministic
(Corollary~\ref{c_cs5}, Theorem~\ref{t_cs6}) and probabilistic
(Theorem~\ref{t_cs7}, Corollary~\ref{c_cs8}). In particular (Remark~\ref{r_cfirst})
we note the discrepancy by a factor $u$ of existing bounds with ours in Corollary~\ref{c_cs5},
\begin{equation*}
|\hat{s}_n-s_n|\leq 3u\,\sum_{k=1}^n{|x_k|}+\mathcal{O}(u^2).
\end{equation*}

\subsection{Models for roundoff error}
Our analyses assume that the inputs $x_k$ are floating point numbers, i.e. 
can be stored exactly without error, and that the summation
produces no overflow or underflow.

Let $u$ denote the unit roundoff.

\paragraph{Standard model \cite{higham2002accuracy}} 
This holds for IEEE floating-point arithmetic and implies that in the absence of underflow or
overflow, the 
operations $\text{op} \in \{+,-,\times, /, \sqrt{}\}$ when applied
to floating point inputs $x$ and $y$ produce
\begin{equation}\label{model:classical}
	\flopt(x \text{ op } y) = (x \text{ op } y)(1 + \delta_{xy}), \qquad |\delta_{xy}| \leq u.
\end{equation}
 In order to derive probabilistic bounds, we treat the roundoffs
as zero-mean random variables, according to two models that differ in
the assumption of independence. 

\begin{model}\label{model:first}
In the computation of interest, the quantities $\delta_{xy}$ in the model \eqref{model:classical} associated with every pair of operands are independent random variables of mean zero. 
\end{model}

\begin{model}\label{model:second}
Let the computation of interest generate rounding errors $\delta_1,\delta_2,\ldots$ in that order. The $\delta_k$ are random variables of mean zero such that \[\E(\delta_k|\delta_1,\ldots,\delta_{k-1}) = \E(\delta_k)  = 0.\]
\end{model}
Higham and Mary \cite{higham2020sharper} note that the mean independence assumption of Model \ref{model:second} is weaker than the independence assumption of Model \ref{model:first}, but stronger than the assumption that the errors are uncorrelated. Connolly, et al.~\cite{connolly2021stochastic} observe that at least one mode of stochastic rounding produces errors satisfying Model \ref{model:second}, but with the weaker bound $|\delta_{xy}|\leq 2u$ in place of \eqref{model:classical}. 

\subsection{Probability theory}
For the derivation of the probabilistic bounds,  
we need a martingale, and a concentration inequality.

\begin{definition}[Martingale \cite{mitzenmacher2005probability}] A sequence of random variables $Z_1,\ldots, Z_n$ is a martingale with respect to the sequence $X_1,\ldots,X_n$ if, for all $k\geq 1$,
	\begin{itemize}
		\item $Z_k$ is a function of $X_1,\ldots,X_k$, 
		\item $\E[|Z_k|]<\infty$, and 
		\item $\E(Z_{k+1}| X_1,\ldots, X_{k}) = Z_{k}$. 
	\end{itemize}
\end{definition}

\begin{lemma}[Azuma-Hoeffding inequality \cite{roch2015modern}]\label{l_azuma}
Let $Z_1,\ldots,Z_n$ be a martingale with respect to a sequence $X_1,\ldots, X_n$, and let $c_k$ be constants with
	\[
		 |Z_{k} - Z_{k-1}| \leq c_k, \qquad 2\leq k\leq n.
	\]
	Then for any $\delta \in (0,1)$, with probability at least $1-\delta$,
	\begin{equation}\label{eqn:Azuma}
		|Z_n-Z_1| \leq  \left(\sum_{k=2}^nc_k^2\right)^{1/2}\sqrt{2\ln(2/\delta)}.
	\end{equation}
	\label{lemma:Azuma}
\end{lemma}

If the bounds on the differences $|Z_k-Z_{k-1}|$ are permitted to fail with a small probability,
then a similar but weaker concentration inequality holds \cite{chung2006concentration}.
Specifically, if $|Z_k-Z_{k-1}|\leq c_k$ simultaneously for all $2\leq k\leq n$ with exceptional probability at most $\eta$, then \eqref{eqn:Azuma} still holds with probability at least $1 - (\delta+\eta)$.

\section{General summation}\label{s_general}
We analyse the error for general summation, presented in Algorithm \ref{alg:sum}.
After representing the summation algorithm as a computational tree,
we present error expressions and a deterministic error bound  (Section~\ref{s_gexact}),
deterministic and probabilistic bounds valid to first order
(Section~\ref{s_gfirst}), followed by two
probabilistic bounds, one for Model~\ref{model:first}  (Section~\ref{s_gprob1}) and one for Model~\ref{model:second} (Section~\ref{s_gprob2}).

\begin{algorithm}
    \centering
    \caption{General Summation \cite[Algorithm~4.1]{higham2002accuracy}} \label{alg:sum}
    \begin{algorithmic}[1]
		\REQUIRE{A set of floating point numbers $\mathcal{S} = \{x_1,\ldots,x_n\}$}
		\ENSURE{$s_n = \sum_{k=1}^n{x_k}$}
		\FOR{$k = 2:n$}
		\STATE{Remove two elements $x$ and $y$ from $\mathcal{S}$}
		\STATE{$s_k = x + y$} \label{line:summation}
		\STATE{Add $s_k$ to $\mathcal{S}$}
		\ENDFOR
    \end{algorithmic}
\end{algorithm}

Denote by $s_k$ the exact partial sum, by $\hat{s}_k$ the computed sum,
and by $e_k=\hat{s}_k-s_k$ the forward error, $2\leq k\leq n$.

\paragraph{Computational tree for Algorithm~\ref{alg:sum}}
We represent the specific ordering of summations in Algorithm~\ref{alg:sum} by
a binary tree with $2n-1$ vertices and the following properties: 
\begin{itemize}
    \item Each vertex represents a partial sum $s_k$ or an input $x_k$.
    \item Each operation $s_k = x+y$ gives rise to two edges $(x,s_k)$ and $(y,s_k)$. 
    \item The root is the output $s_n$, and the leaves are the inputs $x_1,\ldots,x_n$. 
\end{itemize}
Figure \ref{fig:DAG} shows two examples, one for sequential (recursive) summation and one for pairwise summation. 

    \begin{figure}
       \centering
       \begin{minipage}{0.45\textwidth}
         \centering
        		\begin{tikzpicture}[scale=1.2]
		\tikzstyle{every node}+=[inner sep=0pt]
		\fill (0,0) circle (0.08) node [below=5] {$x_1$}; 
		\fill (1,0) circle (0.08) node [below=5] {$x_2$};  
		\fill (1.5,0.75) circle (0.08) node [below=5] {$x_3$}; 
		\fill (2,1.5) circle (0.08) node [below=5] {$x_4$}; 
		\draw (0,0) -- (1.5,2.25) -- (2,1.5); 
		\draw (0.5,0.75)--(1,0); 
		\draw (1.5,0.75) -- (1,1.5); 
	    \fill (0.5,0.75) circle (0.08) node [above left=3] {$s_2$};
		\fill (1,1.5) circle (0.08) node [above left =3] {$s_3$};
		\fill (1.5,2.25) circle (0.08) node [above left=3] {$s_4$};
		\end{tikzpicture}
      \end{minipage}\hfill
      \begin{minipage}{0.45\textwidth}
          \centering
        \begin{tikzpicture}[scale=1.2]
		\tikzstyle{every node}+=[inner sep=0pt]
		\fill (0,0) circle (0.08) node [below=5] {$x_1$}; 
		\fill (1,0) circle (0.08) node [below=5] {$x_2$};  
		\fill (2,0) circle (0.08) node [below=5] {$x_3$}; 
		\fill (3,0) circle (0.08) node [below=5] {$x_4$}; 
		\draw (0,0)-- (0.5,0.75) -- (1,0); 
		\draw (2,0) -- (2.5,0.75) -- (3,0); 
		\draw (0.5,0.75) -- (1.5,1.5) -- (2.5,0.75); 
		\fill (0.5,0.75) circle (0.08) node [above=5] {$s_2$};
		\fill (2.5,0.75) circle (0.08) node [above=5] {$s_3$};
		\fill (1.5,1.5) circle (0.08) node [above=5] {$s_4$};
		\end{tikzpicture}
        \end{minipage}
        \caption{Computational trees for two different summation orderings in
        Algorithm~\ref{alg:sum} when $n=4$.
        Left: sequential (recursive) summation. Right: pairwise summation.}
        \label{fig:DAG}
    \end{figure}
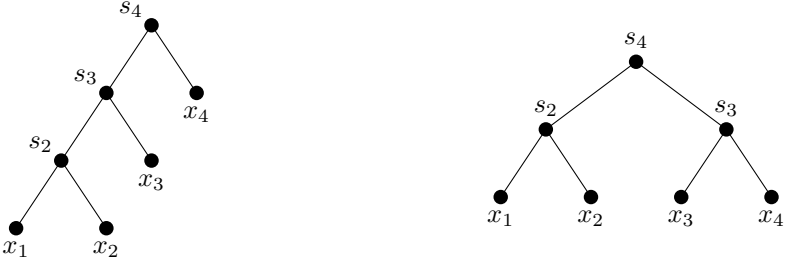

\begin{definition}
Consider the computational tree associated for Algorithm~\ref{alg:sum} applied to $n$ inputs.
\begin{itemize}
    \item The tree defines a partial ordering. 
    If $s_j$ is a descendant of $s_k$, we say $j\prec k$ 
    and $j\preceq k$ if $s_j=s_k$ is possible. 
    \item The height of a node is the length of the longest downward path from that node to a leaf. Leaves have a height of zero. 
    \item The height of the tree is the height of its root. 
    Sequential (recursive) summation yields a tree of height $n-1$, 
    while pairwise summation yields a tree of height $\lceil \ln_2 n\rceil$, 
    see Figure \ref{fig:DAG}.
\end{itemize}
\end{definition}

\subsection{Explicit expressions and deterministic bounds for the error}\label{s_gexact}
We present two expressions for the error in Algorithm~\ref{alg:sum}
(Lemmas \ref{lemma:forwardErrorGenl} and \ref{lemma:forwardErrorRec})
each geared towards a different model for the subsequent probabilistic bounds;
and two deterministic bounds (Theorem~\ref{t_gdet}).

If the computed children are $\hat{x}=x+e_x$ and $\hat{y}=y+e_y$, then 
the computed parent sum, in 
line \ref{line:summation} of Algorithm \ref{alg:sum}, is
\begin{equation*}
    \hat{s}_k = (\hat{x} + \hat{y})(1+\delta_k), \qquad 2\leq k\leq n.
\end{equation*}
Highlight the error in the children,
\begin{align*}
s_k + e_k=\hat{s}_k= ((x + e_x) + (y+e_y))(1+\delta_k)
	=(s_k+e_x+e_y)(1+\delta_k)
	\end{align*}
	to obtain the error in the parent 
	\begin{align}\label{eqn:error1}
	e_k = (s_k + e_x + e_y)(1+\delta_k) - s_k 
	= (e_x + e_y)(1+\delta_k) + s_k\delta_k, \qquad 2\leq k\leq n.
\end{align}
This means the error at node $s_k$ equals the sum of the 
children errors, perturbed slightly, plus a \textit{local error}
that is proportional to the exact partial sum $s_k$.  

The unravelled recursions in Lemma~\ref{lemma:forwardErrorGenl}
generalize the explicit expressions for sequential (recursive) summation in 
\cite[Lemma 3.1]{hallman2021refined},
\begin{equation}\label{e_eric}
e_n=\hat{s}_n-s_n=
\sum_{k=2}^{n}{s_k\delta_k\prod_{\ell=k+1}^{n}{(1+\delta_{\ell})}}
\end{equation}
to general summation algorithms.

We present two expressions that differ in their usefulness under
particular probabilistic models.
The expression in Lemma~\ref{lemma:forwardErrorGenl} works well with Model~\ref{model:first} because the sum is a martingale if summed in {\it reverse order}, see Section~\ref{s_gprob1}. Unfortunately, this approach does not work under the weaker Model~\ref{model:second}. 
This is the reason for Lemma \ref{lemma:forwardErrorRec} where the sum is a martingale under
the original ordering, see Section~\ref{s_gprob2}.

\begin{lemma}[First explicit expression]\label{lemma:forwardErrorGenl}
The error in Algorithm \ref{alg:sum} equals
	\begin{equation}\label{eqn:forwardErrorGenl}
		e_n = \hat{s}_n - s_n = \sum_{k=2}^ns_k\delta_k\prod_{k\prec j \preceq n}(1+\delta_j). 
	\end{equation}
\end{lemma}

This means the forward error is the sum of the local errors at each node, each perturbed by subsequent rounding errors. A similar expression 
is derived in \cite[(4.2)]{higham2002accuracy} 
\begin{equation}\label{eqn:partialError}
	e_n = \sum_{k=2}^n\hat{s}_k\tilde{\delta}_k, 
\end{equation}
but in terms of the computed partial sums $\hat{s}_k$.
By contrast, our \eqref{eqn:forwardErrorGenl} 
depends only on the exact partial $s_k$, which make it  
more amenable to a probabilistic analysis. 

Write \eqref{eqn:error1} in terms of the sum $f_k$
of the children errors at node $s_k$,
\begin{equation*}
    e_k = f_k + (s_k + f_k)\delta_k,\qquad f_k\equiv e_x+e_y,\qquad
    2\leq k\leq n.
\end{equation*}
Unraveling the recurrence gives the second explicit expression for the errors.

\begin{lemma}[Second explicit expression]\label{lemma:forwardErrorRec}
The errors in Algorithm \ref{alg:sum} equal
	\begin{equation}\label{eqn:forwardErrorRec}
		e_k = \hat{s}_k - s_k = \sum_{j\preceq k}(s_j+f_j)\delta_j, \qquad 2\leq k\leq n,
	\end{equation}
and similarly 
	\begin{equation}\label{eqn:frecurrence}
	    f_k = \sum_{j \prec k}(s_j + f_j)\delta_j, \qquad 2\leq k\leq n.
	\end{equation}
\end{lemma}

\begin{theorem}\label{t_gdet}
If $h$ is the height of the computational tree for Algorithm~\ref{alg:sum},
the error is bounded by
\begin{align*}
|e_n| &\leq \sum_{k=2}^n|s_k||\delta_k|\left|\prod_{k\prec j \preceq n}(1+\delta_j)\right| 
\leq u(1+u)^h\sum_{k=2}^n|s_k|\\
&\leq \frac{h^2u^2}{1-hu}\,\sum_{k=1}^n{|x_k|}, \qquad \mathrm{provided}~hu<1.
\end{align*}
\end{theorem}

\begin{proof}
The first bound is a direct consequence of Lemma~\ref{lemma:forwardErrorGenl}.
The second one follows from 
$\sum_{k=2}^n|s_k|\leq h\sum_{k=1}^n|x_k|$,
and $(1+u)^h\leq hu/(1-hu)$ \cite[Lemma 3.1]{higham2002accuracy}.
\end{proof}

A similar bound is derived from \eqref{eqn:partialError} in \cite[(4.3)]{higham2002accuracy},
\[|e_n|\leq u\sum_{k=2}^n|\hat{s}_k|,\] 
is accompanied by the following observation:
\begin{quote}
\textit{In designing or choosing a summation method to achieve high accuracy, the aim should be to minimize the absolute values of the intermediate sums $s_k$.}
\end{quote}
Our probabilistic bounds in
Sections \ref{s_gprob1} and~\ref{s_gprob2} corroborate
this observation, the main difference being a 
dependence on $\left(\sum_k s_k^2\right)^{1/2}$ rather than $\sum_k |s_k|$.

\subsection{Deterministic and probabilistic error bounds valid to first order}\label{s_gfirst}
We present a deterministic bound (\ref{e_gdet1}), and a probabilistic bound 
(Theorem~\ref{thm:firstOrderBound}) under Model~\ref{model:second}, both valid to
first order.

Truncating in Theorem~\ref{t_gdet} yields the first order bound
\begin{equation}\label{e_gdet1}
	e_n = \sum_{k=2}^ns_k\delta_k + \mathcal{O}(h^2u^2),
\end{equation}
which is also derived in \cite[Lemma 2.1]{higham2020sharper}
for the special case sequential (recursive) summation.  

\begin{theorem}\label{thm:firstOrderBound}
Let $\delta_2,\ldots,\delta_n$ be mean-indepen\--dent, zero mean random variables.
Then for any $\delta \in (0,1)$, with probability at least $1-\delta$,
the error in Algorithm \ref{alg:sum} is bounded by
	\begin{equation*}
		|e_n| \leq u\left(\sum_{k=2}^ns_k^2\right)^{1/2}\sqrt{2\ln(2/\delta)} + \mathcal{O}(u^2). \end{equation*}
\end{theorem}

\begin{proof}
With $\delta_1=0$, the sequence 
$Z_1=0$, $Z_i = \sum_{k=2}^i s_k\delta_k$, $2\leq i\leq n$,
is a martingale with respect to $\delta_1,\ldots,\delta_n$, satisfying
\[
    |Z_i - Z_{i-1}| \leq u|s_k|, \qquad 2\leq i \leq n. 
\]
Now apply Lemma~\ref{l_azuma}.
\end{proof}

A closely related bound for sequential (recursive) summation
is derived in \cite[Theorem~2.8]{higham2020sharper} under a model where the inputs $x_k$ are themselves random variables. 
Under this assumption, the probabilistic error bounds become significantly tighter if $\E[x_k] = 0$, leading to the conclusion that centering the data prior to summation can reduce the error, 
see  Section~\ref{s_center}.

A disadvantage of Theorem \ref{thm:firstOrderBound} is 
its uncertainty about the point where the second-order term $\mathcal{O}(u^2)$ 
starts to become relevant. For input dimensions $n > u^{-1}$ or $h>u^{-1}$ in particular, it is desirable to derive bounds that hold to all orders.\\
    
\subsection{General probabilistic bound under Model~\ref{model:first}}\label{s_gprob1}
We derive a bound under the assumption of independent roundoffs $\delta_2,\ldots,\delta_n$.

\begin{theorem}\label{thm:model1Theorem}
Let $\delta_2,\ldots,\delta_n$ be independent zero-mean random variables,
$h$ the height of the computational tree for Algorithm~\ref{alg:sum},
and $\lambda \equiv  \sqrt{2\ln(2n/\eta)}$.
Then for any $\delta \in (0,1)$ and $\eta\in (0,1)$,
with probability at least $1-(\delta+\eta)$,
the error in Algorithm~\ref{alg:sum} is bounded by
	\begin{equation*}
|\hat{s}_n - s_n| \leq u\exp(\lambda\sqrt{h}u)\left(\sum_{k=2}^n{s_k^2}\right)^{1/2}\sqrt{2\ln(2/\delta)}.
	\end{equation*}
\end{theorem}

\begin{proof}
We start by bounding the products in the error expression \eqref{eqn:forwardErrorGenl}. Since 
\begin{equation*}
    \prod_{k\prec j \preceq n}(1+\delta_j)\leq \exp\left(\sum_{k\prec j\preceq n}\delta_j\right), \qquad 2\leq k \leq n-1, 
\end{equation*}
we can apply Lemma~\ref{l_azuma} to the sum on the right to conclude that
for each $k$,  $\eta \in (0,1)$, and $\lambda = \sqrt{2\ln(2/\eta)}$
with failure probability at most $\eta$, the bound 
\begin{equation*}
    \prod_{k\prec j \preceq n}(1+\delta_j)\leq \exp(\lambda\sqrt{h}u)
\end{equation*}
holds. Since the product is nonnegative, we can replace 
the expression on the left with its absolute value and the statement still holds. 
A union bound over at most $n$ such inequalities shows that with failure 
probability at most $n\cdot\eta$ the bound 
\begin{equation*}
    \max_{2\leq k \leq n}\left|\prod_{k\prec j\preceq n}(1+\delta_j)\right| \leq \exp(\lambda \sqrt{h}u)
\end{equation*}
holds. Then substitute $\eta \mapsto \eta/n$ and
$\lambda = \sqrt{2\ln(2n/\eta)}$.

Next, with $\delta_1=0$, the sequence
\begin{equation*}
Z_1=0,\qquad Z_i = \sum_{k=n-i+2}^ns_k\delta_k\prod_{k\prec j \preceq n}(1+\delta_j), 
\qquad 2\leq i \leq n
\end{equation*}
is a martingale with respect to $\delta_n,\delta_{n-1},\ldots,\delta_2$,
that satisfies
\begin{equation*}
    |Z_i - Z_{i-1}| \leq u\exp(\lambda\sqrt{h}u)|s_{n-i+2}|,\qquad 2\leq i \leq n
\end{equation*}
with failure probability at most $\eta$. At last apply Lemma~\ref{l_azuma} with 
additional failure probability $\delta$.
\end{proof}

The bound
    $\left(\sum_{k=2}^n s_k^2\right)^{1/2} \leq \sqrt{h}\sum_{i=1}^n|x_i|$
illustrates that, with high probability, the error in Theorem~\ref{thm:model1Theorem} is bounded by a quantity proportional to the square root of the height of the 
computational tree in Algorithm~\ref{alg:sum}. This suggests that minimizing
the height of the computational tree should reduce the error. 

As long as $\lambda \sqrt{h}u \ll 1$, the probability $\eta$ can be made small, without
much effect on the overall bound in Theorem~\ref{thm:model1Theorem}.

\subsection{General probabilistic bound under Model~\ref{model:second}}\label{s_gprob2}
We derive a bound under the assumption of mean-independent roundoffs $\delta_2, \ldots,\delta_n$.
To this end we rely on the recurrence in \eqref{eqn:frecurrence} and  induction. 

\begin{lemma} \label{lemma:fBound}
Let $\delta_2,\ldots,\delta_n$ be mean-independent mean-zero random variables,
$h$ the height of the computational tree for Algorithm~\ref{alg:sum},
and $\lambda \equiv  \sqrt{2\ln(2n/\eta)}$.
For any $\eta > 0$ with $\lambda\sqrt{h}u < 1$,  the bounds
\begin{equation}\label{eqn:inductiveBound}
    |f_k| \leq \frac{\lambda u}{1-\lambda \sqrt{h}u} 
    \left(\sum_{j\prec k}s_j^2\right)^{1/2}, \qquad 2\leq k\leq n
\end{equation}
all hold simultaneously with probability at least $1-\eta$.
\end{lemma}

\begin{proof}
The proof is by induction, and in terms of the failure probability $\eta$. 

\paragraph{Induction basis $k=2$}
The children of $s_2$ are both leaves and 
floating point numbers,
hence $f_2 = 0$ and \eqref{eqn:inductiveBound} holds deterministically. 
    
\paragraph{Induction hypothesis}
Assume that \eqref{eqn:inductiveBound} 
holds simultaneously for all $2\leq j \leq k-1$ with failure probability at most $(k-1)\eta/n$. 
\paragraph{Induction step}
Write \eqref{eqn:frecurrence} as 
    \begin{equation*}
        f_k = \sum_{j=2}^{k-1}(s_j+f_j)\delta_j\mathds{1}_{j\prec k}.
    \end{equation*}
With $\delta_1= 0$, the sequence  
$Z_1=0$, $Z_i = \sum_{j=2}^i(s_j + f_j)\delta_j\mathds{1}_{j\prec k}$
   is a martingale with respect to $\delta_1,\ldots, \delta_{k-1}$, 
   and by the induction hypothesis satisfies 
   \begin{equation*}
       |Z_i - Z_{i-1}| \leq \begin{cases} us_i + \frac{\lambda u^2}{1-\lambda\sqrt{h}u}\left(\sum_{j\prec i}s_j^2\right)^{1/2} & i \prec k,\\
       0 & i \nprec k, 
       \end{cases},\qquad  2\leq i \leq k-1, 
   \end{equation*}
with failure probability at most $(k-1)\eta/n$. 
Lemma~\ref{l_azuma} then implies that with additional failure probability $\eta/n$, 
   \begin{align*}
       |f_k| &\leq \lambda u \left(\sum_{j\prec k}\left(s_j + \frac{\lambda u}{1-\lambda \sqrt{h}u}\Big(\sum_{\ell\prec j}s_\ell^2\Big)^{1/2} \right)^2\right)^{1/2}\\
       &\leq \lambda u \left(\sum_{j\prec k}s_j^2\right)^{1/2} + \lambda u \left(\sum_{j\prec k}\left(\frac{\lambda u}{1-\lambda \sqrt{h}u}\Big(\sum_{\ell\prec j}s_\ell^2\Big)^{1/2} \right)^2\right)^{1/2}
       \end{align*}
       where the second inequality follows from the two-norm triangle inequality.
Exploit the fact that $s_j$ can appear at most $h$ times,
once for itself and once for each of its ancestors, to bound 
the second summand by
       \begin{align*}
  \lambda u \left(\sum_{j\prec k}\left(\frac{\lambda u}{1-\lambda \sqrt{h}u}\Big(\sum_{\ell\prec j}s_\ell^2\Big)^{1/2} \right)^2\right)^{1/2}
&= \frac{\lambda^2u^2}{1-\lambda \sqrt{h}u}\left(\sum_{j\prec k}\sum_{\ell\prec j}s_\ell^2\right)^{1/2}\\
       &\leq 
        \frac{\lambda^2u^2}{1-\lambda \sqrt{h}u}\left(h\sum_{j\prec k}s_j^2\right)^{1/2}.
       \end{align*}
       
       Combine everything 
       \begin{align*}
       |f_k| &\leq \lambda u \left(\sum_{j\prec k}s_j^2\right)^{1/2} +
   \frac{\lambda^2u^2}{1-\lambda \sqrt{h}u}\left(h\sum_{j\prec k}s_j^2\right)^{1/2}\\
       &= \lambda u\left(1 + \frac{\lambda \sqrt{h}u}{1-\lambda \sqrt{h}u}\right) \left(\sum_{j\prec k}s_j^2\right)^{1/2}
       = \frac{\lambda u}{1-\lambda \sqrt{h}u} \left(\sum_{j\prec k}s_j^2\right)^{1/2}.
   \end{align*}
 Thus \eqref{eqn:inductiveBound} holds simultaneously for all $2\leq j \leq k$ with failure probability at most $k\eta/n$. By induction
 \eqref{eqn:inductiveBound} holds for all $2\leq k\leq n$ with failure probability at most $\eta$. \end{proof}

Below is the final probabilistic bound under Model~\ref{model:second}.

\begin{theorem}\label{thm:model2Theorem}
Let $\delta_2,\ldots,\delta_n$ be mean-independent mean-zero random variables,
$h$ the height of the computational tree for Algorithm~\ref{alg:sum},
and $\lambda \equiv  \sqrt{2\ln(2n/\eta)}$.
For any $\eta > 0$ with $\lambda\sqrt{h}u < 1$ and $\delta \in (0,1)$,
with probability at least $1-(\delta+\eta)$,
the error in Algorithm~\ref{alg:sum} is bounded by
\begin{equation}\label{eqn:model2Bound}
    |\hat{s}_n -s_n| \leq \frac{u}{1-\lambda \sqrt{h}u}\left(\sum_{k=2}^n s_k^2\right)^{1/2}\sqrt{2\ln(2/\delta)}.
\end{equation}
\end{theorem}

\begin{proof}
Apply Lemma~\ref{l_azuma} to (\ref{eqn:forwardErrorRec})
with failure probability $\delta$, and bound $|f_k|$ with Lemma~\ref{lemma:fBound}.
\end{proof}

The bound in Theorem~\ref{thm:model2Theorem}
is nearly identical to the one in Theorem \ref{thm:model1Theorem}, 
the difference becoming significant only when $\lambda \sqrt{h}u$ is close to 1.

\input{Centering}
\input{Compensated}
\input{Experiments}

\subsection*{Acknowledgement}
We thank Johnathan Rhyne for helpful discussions.
\bibliographystyle{siamplain}
\bibliography{references}

\end{document}

%% file: ex_shared.tex
\usepackage{amsmath}
\usepackage{amssymb}
\usepackage{graphicx}
\usepackage{arydshln}
\usepackage{hyperref}
\usepackage{enumitem}
\usepackage{comment}
\usepackage{algorithmic}
\usepackage{dsfont}
\usepackage{tikz}

\newcommand{\E}{\mathbb{E}}

\DeclareMathOperator*{\flopt}{fl}

\definecolor{cerulean}{rgb}{0.16, 0.32, 0.75}

\newtheorem{model}{Model}[section]
\newtheorem{remark}{Remark}[section]

\numberwithin{theorem}{section}

\newcommand{\TheTitle}{Deterministic and Probabilistic Error Bounds for 
Floating Point Summation Algorithms} 
\newcommand{\TheAuthors}{Eric Hallman and Ilse C.F. Ipsen}

\headers{Error bounds for Floating Point Summation}{\TheAuthors}

\title{{\TheTitle}\thanks{This research was supported in part by the National Science Foundation through grant DMS-1745654 and DMS-1760374.}}

\author{
  Eric Hallman\thanks{Department of Mathematics, North Carolina State University (\email{erhallma@ncsu.edu})} \and
  Ilse C.F. Ipsen\thanks{Department of Mathematics, North Carolina State University (\email{ipsen@ncsu.edu})}
}

%% file: Centering.tex
\section{Shifted General Summation}\label{s_center}
After discussing the motivation for shifted summation, we present 
Algorithm~\ref{alg:shiftSum} for shifted general summation; and
derive error expressions, deterministic and probabilistic bounds 
for shifted sequential (recursive) summation (Section~\ref{s_ss})
and shifted general summation (Section~\ref{s_sg}).

Shifted summation is motivated by 
work  in computer architecture \cite{DSC19,CDRS21} and formal methods
for program verification \cite{Lohar19}
where not only the roundoffs  but also the inputs are interpreted as 
random variables sampled from 
 some distribution. Then one can compute statistics for
 the total roundoff error and estimate the probability that it is bounded by $tu$
 for a given $t$.
 
Adopting instead the approach for probabilistic roundoff error analysis in \cite{higham2019new,ipsen2020probabilistic},
probabilistic bounds for random data are derived in 
\cite{higham2020sharper}, with improvements in \cite{hallman2021refined}.
The bound below suggests that sequential summation is accurate if the summands 
$x_k$ are tightly clustered around zero. 

\begin{lemma}[Theorem 5.2 in \cite{hallman2021refined}]\label{l_th52}
Let the roundoffs $\delta_k$ be independent mean-zero random variables, 
and the summands $x_k$ be independent random variables with mean~$\mu$ and 
`variance' $\max_{1\leq k\leq n}{|x_k-\mu|}\leq\sigma$. 
Then for any $0<\delta<1$, with probability at least $1-\delta$,
the error in sequential (recursive) summation is bounded by
\begin{equation*}
|\hat{s}_n-s_n|\leq (1+\gamma) 
(\lambda\,n^{3/2}\, |\mu| + \lambda^2 n \,\sigma)\, u
\end{equation*}
where $\lambda=\sqrt{2\ln{(6/\delta)}}$ and 
$\gamma=\exp{(\lambda\,(\sqrt{n}u+nu^2)/(1-u))}-1$
\end{lemma}

This observation has been applied to non-random data to reduce the backward error
 \cite[Section 4]{higham2020sharper}. 
The shifted algorithm for sequential (recursive) summation 
\cite[Algortihm 4.1]{higham2020sharper} is extended to general summation
in Algorithm~\ref{alg:shiftSum}.

\begin{algorithm}\caption{Shifted General Summation}\label{alg:shiftSum}
\begin{algorithmic}[1]
\REQUIRE Summands $x_k$, $1\leq k \leq n$, shift $c$
\ENSURE $s_n = \sum_{k=1}^n{x_k}$
\FOR{$k=1:n$}
\STATE $y_k=x_k-c$
\ENDFOR
\STATE $t_n$ = output of Algorithm \ref{alg:sum} on $\{y_1,\ldots,y_n\}$
\RETURN $s_n = t_n +nc$
\end{algorithmic}
\end{algorithm}

Note that the pseudo-code in Algorithm \ref{alg:shiftSum} is geared 
towards exposition. In practice, one shifts the $x_k$ immediately prior to the
summation, to avoid allocating additional storage for the $y_k$.
The ideal but impractical choice for centering is the empirical mean $c=s_n/n$.
A more practical approximation is $c=(\min_k{x_k}+\max_k{x}_k)/2$.

\subsection{Shifted sequential (recursive) summation}\label{s_ss}
We present an expression for the error (\ref{e_s1}),
a deterministic bound (\ref{e_s2}), and a probabilistic bound 
under Model~\ref{model:first} (Theorem~\ref{t_sprob}).

We incorporate the final uncentering operation as the $(n+1)$st summation, and 
describe shifted sequential summation in exact arithmetic as
\begin{equation*}
t_k=\sum_{j=1}^k{y_j}=t_{k-1}+y_k, \quad 1\leq k\leq n+1, \qquad s_n=t_{n+1}=t_n+nc,
\end{equation*}
where $y_k=x_k -c$, $1\leq k\leq n$ and $y_{n+1}=nc$. In floating point arithmetic, the $\epsilon_j$ denote the (un)centering errors,
and the $\delta_j$ the summation errors.
\begin{eqnarray*}
\hat{t}_1&= &y_1 (1+\epsilon_1)(1+\delta_1), \qquad \delta_1=0\\
\hat{t}_k &= &\left(\hat{t}_{k-1}+y_k(1+\epsilon_k)\right)(1+\delta_k), \qquad 2\leq k\leq n+1.
\end{eqnarray*}
The idea behind shifted summation for non-random data is that the centered sum
\begin{equation*}
|\hat{t}_n-t_n|\leq n u\sum_{k=1}^n{|y_k|}+\mathcal{O}(u^2)
\end{equation*}
can be more accurate if centering reduces the magnitude of the partial sums and the
summands, see also the discussions after Theorems~\ref{t_gdet} and~\ref{t_sprob}.

The error expression corresponding to (\ref{e_eric}) and 
Lemma~\ref{lemma:forwardErrorGenl} is
\begin{equation}\label{e_s1}
e_n=\hat{s}_n-s_n=
\underbrace{\sum_{k=2}^{n+1}{t_k\delta_k\prod_{\ell=k+1}^{n+1}{(1+\delta_{\ell})}}}_{\text{Summation}}
+\underbrace{\sum_{k=1}^{n+1}{y_k\epsilon_k\prod_{\ell=k}^{n+1}{(1+\delta_{\ell})}}}_{\text{Centering}}
\end{equation}
with the deterministic bound
\begin{equation}\label{e_s2}
|e_n|\leq u(1+u)^n\left(\sum_{k=2}^n{|s_k-kc|}+\sum_{k=1}^n{|x_k-c|}+|s|+|nc|\right).
\end{equation}

If all $\epsilon_k=0$ and $\delta_{n+1}=0$, then $t_k=s_k$ and the error (\ref{e_s1}) 
reduces to that for plain sequential summation (\ref{e_eric}). 

The probabilistic bound below can be considered to be a special case of the general bound in
Theorem~\ref{t_sprob3}, but it requires only a single probability.

\begin{theorem}[Probabilistic bound for shifted sequential (recursive) summation]\label{t_sprob}
Let  $\delta_j$ and  $\epsilon_j$ be independent mean-zero random variables,
and $\gamma_k\equiv (1+u)^k-1$.
Then for any $0<\delta<1$, with probability at least $1-\delta$,
\begin{equation*}
|\hat{s}_n-s_n|\leq \max_{1\leq k\leq n+1}{\left(|s_k-kc|+|x_k-c|\right)} \sqrt{\frac{u\,\gamma_{2(n+2)}}{2}}\,\sqrt{2\ln(2/\delta)}.
\end{equation*}
For $n\ll1/u$ we have
\begin{equation*}
\sqrt{\frac{u\gamma_{2(n+2)}}{2}}\leq \sqrt{\frac{(n+2)u^2}{1-2(n+2)u}}\approx \sqrt{n+2}\,u.
\end{equation*}
\end{theorem}

\begin{proof}
We construct a Martingale by extracting the accumulated
roundoffs in a last-in first-out manner. Set $Z_0\equiv 0$,
\begin{eqnarray*}
Z_1 &=& t_{n+1}\delta_{n+1} + y_{n+1}\epsilon_{n+1}(1+\delta_{n+1}) = 
s \delta_{n+1} +nc\,\epsilon_{n+1}(1+\delta_{n+1})\\
Z_k &=& \sum_{j=n-k+2}^{n+1}{t_j\delta_j\prod_{\ell=j+1}^{n+1}{(1+\delta_{\ell})}}
+\sum_{j=n-k+2}^{n+1}{y_j\epsilon_j\prod_{\ell=j}^{n+1}{(1+\delta_{\ell})}}, \qquad 1\leq k\leq n\\
Z_{n+1} &=& \sum_{j=2}^{n+1}{t_j\delta_j\prod_{\ell=j+1}^{n+1}{(1+\delta_{\ell})}}
+\sum_{j=1}^{n+1}{y_j\epsilon_j\prod_{\ell=j}^{n+1}{(1+\delta_{\ell})}}=\hat{s}_n-s_n.
\end{eqnarray*}
By construction, $Z_1$ is a function of $\xi_{n+1}\equiv(\delta_{n+1}, \epsilon_{n+1})$,
while  $Z_k$ is a function of $\xi_j\equiv(\delta_j, \epsilon_j)$ for $n-k+2\leq j\leq n+1$,
and finally $Z_{n+1}$ is a function of all $\xi_j\equiv(\delta_j, \epsilon_j)$ for $1\leq j\leq n+1$.

By assumption $\eta_j\equiv(\delta_j, \epsilon_j)$ are tuples of $2(n+1)$ 
independent random variables with
$\E[\delta_j]=0=\E[\epsilon_j]$ and $|\delta_j|,|\epsilon_j|\leq u$, $1\leq j\leq n+1$.
Then
\begin{equation*}
\E[Z_{k+1}|\xi_{n-k+2}, \ldots, \xi_{n+1}]=Z_k, \qquad 0\leq k\leq n.
\end{equation*}
Combining all of the above shows that $Z_0,\ldots, Z_{n+1}$ is a Martingale in 
$\xi_{n+1},\ldots, \xi_1$.

The differences are  for $2\leq k\leq n$
\begin{eqnarray*}
Z_1-Z_0 &=& t_{n+1}\delta_{n+1} + y_{n+1}\epsilon_{n+1}(1+\delta_{n+1})\\
Z_k-Z_{k-1}&=&t_{n-k+2}\delta_{n-k+2}\prod_{\ell=n-k+3}^{n+1}{(1+\delta_{\ell})}+
y_{n-k+2}\epsilon_{n-k+2}\prod_{\ell=n-k+2}^{n+1}{(1+\delta_{\ell})},\\
Z_{n+1}-Z_n &=& y_1\epsilon_1\prod_{\ell=1}^{n+1}{(1+\delta_{\ell})}.
\end{eqnarray*}
They are bounded by
\begin{eqnarray*}
|Z_1-Z_0|&\leq & c_1\equiv u(1+u)\,(|t_{n+1}|+|y_{n+1}|)=u(1+u)(|s|+|nc|)\\
|Z_k-Z_{k-1}|&\leq & c_k\equiv u(1+u)^{k-1}\,(|t_{n-k+2}|+|y_{n-k+2}|), \qquad 2\leq k\leq n\\
|Z_{n+1}-Z_n| &\leq & c_{n+1}\equiv  u\, (1+u)^n\,|y_1|
\end{eqnarray*}
Lemma~\ref{l_azuma} implies that for any $0<\delta<1$ with probability at least $\delta$
\begin{equation*}
|Z_{n+1}-Z_0|\leq \sqrt{\sum_{k=1}^{n+1}{c_k^2}}\sqrt{2\ln(2/\delta)}.
\end{equation*}
Here $|Z_{n+1}-Z_0|=|\hat{s}_n-s_n|$. With 
$v_1=|y_1|$, $v_k=|t_k|+|y_k|$, $2\leq k\leq n+1$, the sum under the
square root is bounded by 
\begin{equation*}
\sum_{k=1}^{n+1}{c_k^2}\leq\sum_{k=1}^{n+1}{u^2(1+u)^{2k}\,|v_{n+2-k}|^2}
\leq\max_{1\leq j\leq n+1}{|v_j|^2}{u^2\sum_{k=1}^{n+1}{(1+u)^{2k}}}
\end{equation*}
The trailing term is a geometric sum 
\begin{equation*}
\sum_{k=1}^{n+1}(1+u)^{2k}\leq \frac{(1+u)^{2(n+2)}-1}{(1+u)^2-1}=
\frac{\gamma_{2n}}{u^2+2u}
\end{equation*}
Combining everything gives
$\sqrt{\sum_{k=1}^{n+1}{c_k^2}}\leq\max_{1\leq j\leq n+1}{|v_j|}\sqrt{\frac{u\gamma_{2(n+2)}}{2}}$.
The inequality for $\sqrt{u\gamma_{2(n+2)}/2}$ follows from
\cite[Lemma 3.1]{higham2002accuracy}.
\end{proof}

Theorem~\ref{t_sprob} suggests that the error has a bound proportional to $\sqrt{n+2}$, and
that shifted sequential summation 
improves the accuracy over plain sequential summation if
\begin{equation*}
\max_{k}{\left(|s_k-kc|+|x_k-c|\right)}\ll \max_{k}{\left(|s_k|+|x_k|\right)},
\end{equation*}
that is, if shifting reduces the magnitude of the partial sums. This is confirmed by 
numerical experiments. However, the experiments also illustrate that shifted summation can 
decrease the accuracy in the summation of naturally centered data, such as $x_k$
being normal(0,1) random variables. 

\subsection{Shifted general summation}\label{s_sg}
We present a probabilistic bound for shifted general summation
under Model~\ref{model:second}, based on the following
computational tree.
If Algorithm \ref{alg:sum} operates with a tree of height $h$, then the 
analogous version in Algorithm \ref{alg:shiftSum} 
operates with a tree of height $h+2$ on the $2n+1$ input data 
\[x_1, -c, x_2, -c, \ldots, x_n, -c, nc\]
The leading $2n$ inputs are centered in pairs $y_k=x_k-c$ at the beginning, 
the intermediate sums are $t_k$,
and the final uncentering is $s_n=t_n+nc$.

\begin{theorem}[Probabilistic bound for shifted general summation]\label{t_sprob2}
Let $\delta_k$ be mean-independent mean-zero random variables,
$h$ the height of the computational tree in the inner call to Algorithm~\ref{alg:sum},
and $\lambda \equiv  \sqrt{2\ln(2(2n+1)/\eta)}$.
For any $\eta > 0$ with $\lambda\sqrt{h+2}u < 1$ and $\delta \in (0,1)$,
with probability at least $1-(\delta+\eta)$,
the error in Algorithm~\ref{alg:shiftSum} is bounded by
\begin{equation*}
    |\hat{s}_n - s_n| \leq \frac{u}{1-\lambda \sqrt{h+2}u}\left(s_n^2 + \sum_{k=2}^n t_k^2 + \sum_{k=1}^n y_k^2 \right)^{1/2}\sqrt{2\ln(2/\delta)}.
\end{equation*}
\end{theorem}

\begin{proof}
This is a direct consequence of Theorem \ref{thm:model2Theorem}. 
\end{proof}

A similar bound under Model \ref{model:first} follows from Theorem \ref{thm:model1Theorem}. 

\begin{theorem}[Probabilistic bound for shifted general summation]\label{t_sprob3}
Let $\delta_k$ be independent mean-zero random variables,
$h$ the height of the computational tree in the inner call to Algorithm~\ref{alg:sum},
and $\lambda \equiv  \sqrt{2\ln(2(2n+1)/\eta)}$.
For any $\eta > 0$ with $\lambda\sqrt{h+2}u < 1$ and $\delta \in (0,1)$,
with probability at least $1-(\delta+\eta)$,
the error in Algorithm~\ref{alg:shiftSum} is bounded by
\begin{equation*}
    |\hat{s}_n - s_n| \leq u\exp(\lambda \sqrt{h+2}u)\left(s_n^2 + \sum_{k=2}^n t_k^2 + \sum_{k=1}^n y_k^2 \right)^{1/2}\sqrt{2\ln(2/\delta)}.
\end{equation*}
\end{theorem}

%% file: Compensated.tex
\section{Compensated sequential summation}\label{s_compensated}
We analyse the forward error for compensated sequential summation.

After presenting compensated summation 
(Algorithm~\ref{alg:compensated}) and the roundoff error model (\ref{e_model1}), 
we  derive three different expressions for the compensated summation error
(Section~\ref{s_cexact}), followed by 
bounds that are exact to first order (Section~\ref{s_cfirst})
and to second order (Section \ref{s_csecond}).

Algorithm~\ref{alg:compensated} is the formulation \cite[Theorem 8]{goldberg1991every}
of the `Kahan Summation Formula' \cite{kahan1965pracniques}. A version with opposite signs 
is presented in \cite[Algorithm 4.2]{higham2002accuracy}.

\begin{algorithm}\caption{Compensated Summation \cite[Theorem 8]{goldberg1991every} \cite[page 9-4]{Kahan73}}\label{alg:compensated}
\begin{algorithmic}[1]
\REQUIRE Summands $x_k$, $1\leq k\leq n$
\ENSURE $s_n = \sum_{k=1}^n{x_k}$
\STATE $s_1=x_1$, $c_1=0$
\FOR{$k=2:n$}
\STATE $y_k=x_k-c_{k-1}$
\STATE $s_k=s_{k-1}+y_k$
\STATE $c_k=(s_k-s_{k-1}) -y_k$
\ENDFOR
\RETURN $s_n$
\end{algorithmic}
\end{algorithm}

Following \cite[page 9-5]{Kahan73}, our finite precision model of 
Algorithm~\ref{alg:compensated} is 
\begin{align}\label{e_model1}
\begin{split}
    \hat{s}_1 &= s_1=x_1, \qquad \hat{c}_1= 0 \\
\hat{y}_k&= (x_k-\hat{c}_{k-1})(1+\eta_k), \qquad 2\leq k\leq n, \qquad \eta_2=0\\
\hat{s}_k &= (\hat{s}_{k-1} +\hat{y}_k)(1+\sigma_k)\\
 \hat{c}_k&= \left((\hat{s}_k-\hat{s}_{k-1})(1+\delta_k) - 
 \hat{y}_k\right)(1+\beta_k),
\end{split}
\end{align}

\subsection{Error expressions}\label{s_cexact}
We express the accumulated roundoff in Algorithm~\ref{alg:compensated}
as a recursive matrix vector product (Lemma~\ref{l_cs1}),
then unravel the recursions into an explicit form (Theorem~\ref{thm:compExpr1}),
followed by two other exact expressions derived in a different
way (Theorem~\ref{t_cs4}). 

\begin{lemma}\label{l_cs1}
With assumptions (\ref{e_model1}),  and
\begin{equation*}
\Gamma_k\equiv(1+\sigma_k)(1+\delta_k), \qquad 
\Psi_k\equiv (1+\delta_k)(1+\beta_k), \qquad 
2\leq k\leq n,
\end{equation*}
define
\begin{equation*}
P_k\equiv \begin{bmatrix} 1+\sigma_k & -(1+\eta_k)(1+\sigma_k)\\
\sigma_k\Psi_k & (1+\eta_k)(1-\Gamma_k)(1+\beta_k)
\end{bmatrix}, \qquad 2\leq k\leq n.
\end{equation*}
Then
\begin{equation*}
\begin{bmatrix}e_k\\\hat{c}_k\end{bmatrix} =
P_k\begin{bmatrix}e_{k-1}\\ \hat{c}_{k-1}\end{bmatrix}+
P_k\begin{bmatrix}s_{k-1}\\ -x_k\end{bmatrix}+\begin{bmatrix}-s_k\\0\end{bmatrix},
\qquad 2\leq k\leq n.
\end{equation*}
\end{lemma}

\begin{proof}
Since $e_1=\hat{c}_1=0$,
\begin{equation*}
\begin{bmatrix}e_2\\c_2\end{bmatrix}=P_2\begin{bmatrix}x_1\\-x_2\end{bmatrix}
+\begin{bmatrix}-s_2\\0\end{bmatrix}.
\end{equation*}
For the subsequent recursions abbreviate
\begin{eqnarray*}
p_{11}&=&1+\sigma_k, \quad p_{12}=-(1+\eta_k)(1+\sigma_k), \quad
p_{21}=\sigma_k\Psi_k=\sigma_k(1+\delta_k)(1+\beta_k),\\
p_{22}&=&(1+\eta_k)(1-\Gamma_k)(1+\beta_k)=
(1+\eta_k)\left(1-(1+\sigma_k)(1+\delta_k)\right)(1+\beta_k)
\end{eqnarray*}
To derive the recurrence for $e_k$, write
\begin{equation}\label{e_aux}
\hat{s}_k=(\hat{s}_{k-1}+\hat{y}_k)(1+\sigma_k)=(e_{k-1}+s_{k-1}+\hat{y}_k)(1+\sigma_k)
\end{equation}
Thus 
\begin{eqnarray*}
e_k&=&\hat{s}_k-s_k=(e_{k-1}+s_{k-1})(1+\sigma_k)-
(\hat{c}_{k-1}-x_k)(1+\eta_k)(1+\sigma_k)-s_k\\
&=& p_{11}(e_{k-1}+s_{k-1})+p_{12}(\hat{c}_{k-1}-x_k)-s_k.
\end{eqnarray*}
To derive the recurrence for $\hat{c}_k$, start with (\ref{e_aux})
\begin{equation*}
\hat{s}_k-\hat{s}_{k-1}=\hat{s}_{k-1}\sigma_k+\hat{y}_k(1+\sigma_k)
=(e_{k-1}+s_{k-1})\sigma_k +\hat{y}_k(1+\sigma_k).
\end{equation*}
Multiply by $1+\delta_k$ and subtract $\hat{y}_k=x_k-\hat{c}_{k-1}$
\begin{eqnarray*}
(\hat{s}_k-\hat{s}_{k-1})(1+\delta_k)-\hat{y}_k&=&
(e_{k-1}+s_{k-1})\sigma_k(1+\delta_k) +
\hat{y}_k(\underbrace{(1+\sigma_k)(1+\delta_k)}_{\Gamma_k}-1)\\
&=&(e_{k-1}+s_{k-1})\sigma_k(1+\delta_k) +
(\hat{c}_{k-1}-x_k)(1+\eta_k)(1-\Gamma_k).
\end{eqnarray*}
Multiply by $1+\beta_k$ to obtain
$\hat{c}_k=p_{21}(e_{k-1}+s_{k-1}) +p_{22}(\hat{c}_{k-1}-x_k)$.
\end{proof}

Now we unravel this recursion, so that the matrix product ends
with a matrix $\widetilde{P}_j$ all of whose elements are $\mathcal{O}(u)$.
Note the correspondence between \eqref{eqn:compErr} and the error expression in \eqref{e_eric}. 

\begin{theorem}\label{thm:compExpr1}
With assumptions (\ref{e_model1}),
the recursions in Lemma \ref{l_cs1} have the explicit form
\begin{equation}\label{eqn:compErr}
        \begin{bmatrix}
    e_k\\\hat{c}_k
    \end{bmatrix}
    = \sum_{j=2}^{k-1}(P_k\cdots P_{j+1})\widetilde{P}_j\begin{bmatrix}
    s_j \\ x_j\end{bmatrix}+
    \widetilde{P}_k\begin{bmatrix}s_k\\x_k\end{bmatrix}, 
    \qquad 2 \leq k\leq n,
    \end{equation}
where
\begin{align*}
\widetilde{P}_j
     \equiv \begin{bmatrix}
    \sigma_j & \eta_j(1+\sigma_j)\\
    \sigma_j\Psi_j & (1+\beta_j)(\delta_j + \eta_j(\Gamma_j-1)).
    \end{bmatrix}=\mathcal{O}(u), \qquad 2\leq j\leq n.
\end{align*}
\end{theorem}

\begin{proof}
Write the last two summands in  the recursion in Lemma~\ref{l_cs1} as
\begin{align*}
P_k\begin{bmatrix}s_{k-1}\\-x_k\end{bmatrix}+\begin{bmatrix}-s_k\\ 0 \end{bmatrix}
=P_k\begin{bmatrix}1&-1\\ 0&-1\end{bmatrix}\begin{bmatrix}s_k\\ x_k\end{bmatrix}
-\begin{bmatrix}1 &0\\0&0\end{bmatrix}\begin{bmatrix}s_k\\ x_k\end{bmatrix}=
\widetilde{P}_k\begin{bmatrix}s_{k}\\ x_k\end{bmatrix},
\end{align*}
where 
\begin{align*}
\begin{split}
        \widetilde{P}_k = P_k\begin{bmatrix}
    1 & -1 \\ 0 & -1
    \end{bmatrix} - \begin{bmatrix}
    1 & 0 \\ 0 & 0 
    \end{bmatrix}
    = \begin{bmatrix}
    \sigma_k & \eta_k(1+\sigma_k)\\
    \sigma_k\Psi_k & (1+\beta_k)(\delta_k + \eta_k(\Gamma_k-1))
    \end{bmatrix}.
\end{split}
\end{align*}
Then the recursion in Lemma~\ref{l_cs1} is identical to
\begin{align*}
\begin{bmatrix}e_k\\\hat{c}_k\end{bmatrix} =
P_k\begin{bmatrix}e_{k-1}\\ \hat{c}_{k-1}\end{bmatrix}+
\widetilde{P}_k\begin{bmatrix}s_{k}\\ x_k\end{bmatrix},
\qquad 2\leq k\leq n.
\end{align*}
Now unravel the recursion and terminate with $e_1=\hat{c}_1=0$.
\end{proof}

To strengthen the justification for the first order bound
in Section~\ref{s_cfirst}, we present two 
additional exact expressions for the error that 
are independent of Theorem~\ref{thm:compExpr1}.

\begin{theorem}\label{t_cs4}
With assumptions (\ref{e_model1}), 
the error in Algorithm~\ref{alg:compensated} equals 
\begin{align*}
e_n = s_n\sigma_n +\sum_{j=4}^n{x_j\eta_j\prod_{k=j}^n{(1+\sigma_k)}}
+\sum_{j=2}^{n-1}{\left(s_j\sigma_j-\hat{c}_j(1+\eta_{j+1})\right)
\prod_{k=j+1}^n{(1+\sigma_k)}}
\end{align*}
and
\begin{align*}
e_n= s_n\,\sigma_n+\left(X_n+E_{n-1}\right)(1+\sigma_n),
\end{align*}
where
\begin{eqnarray*}
X_n&\equiv&x_n\eta_n(1+\beta_n)+\sum_{j=2}^{n-1}{x_j(\eta_j-\delta_j)
\prod_{\ell=j}^{n-1}{(1+\beta_{\ell})(1+\eta_{\ell+1})}}\\
\Theta_k &\equiv& 1-(1+\delta_k)(1+\beta_k)(1+\eta_{k+1}),\qquad 2\leq k\leq n-1\\
E_{n-1}&\equiv& e_{n-1}\Theta_{n-1}+\sum_{j=2}^{n-2}{e_j(\Theta_j+
\delta_{j+1})\prod_{\ell=j+1}^{n-1}{(1+\beta_{\ell})(1+\eta_{\ell+1})}}.
\end{eqnarray*}
\end{theorem}

\subsection{First order bound}\label{s_cfirst}
We present a first-order expression and bound for the error in 
Algorithm~\ref{alg:compensated} and note the discrepancy with existing bounds
(Remark~\ref{r_cfirst}.

\begin{corollary}\label{c_cs5}
With assumptions (\ref{e_model1}), Algorithm~\ref{alg:compensated} 
has the first order error
\begin{equation*}
e_n=\hat{s}_n-s_n=s_n\,\sigma_n+x_n\eta_n+\sum_{j=2}^{n-1}{x_j(\eta_j-\delta_j)}
+\mathcal{O}(u^2),
\end{equation*}
with the bound
\begin{equation*}
|\hat{s}_n-s_n|\leq 3u\,\sum_{j=1}^{n}{|x_j|}+\mathcal{O}(u^2).
\end{equation*}
 \end{corollary}

\begin{proof}
This follows from the three expressions in Theorems \ref{thm:compExpr1}
and~\ref{t_cs4}, but is most easily seen from the second expression 
in Theorem~\ref{t_cs4}. There the first order error is extracted from the summands 
$s_n\sigma_n +X_n(1+\sigma_n)$,
where 
\begin{align*}
X_n&=x_n\eta_n(1+\beta_n)+\sum_{j=2}^{n-1}{x_j(\eta_j-\delta_j)
\prod_{\ell=j}^{n-1}{(1+\beta_{\ell})(1+\eta_{\ell+1})}}\\
&=x_n\eta_n+\sum_{j=2}^{n-1}{x_j(\eta_j-\delta_j)}+\mathcal{O}(u^2).
\end{align*}
\end{proof}

Corollary~\ref{c_cs5} suggests that the errors in the 'correction' steps
3 and~5 of Algorithm~\ref{alg:compensated} dominate the first order error. 

\begin{remark}\label{r_cfirst}
The existing bounds in \cite[Section 4.3, Appendix]{goldberg1991every}, \cite[(4.9)]{higham2002accuracy},
\cite[page 9-5]{Kahan73},
and \cite[Section 4.2.2]{Knuth2} are not consistent
with Corollary~\ref{c_cs5}, which implies that the computed sum equals
\begin{equation*}
\hat{s}_n=s_n\,(1+\sigma_n)+\sum_{j=2}^{n-1}{x_j(\eta_j-\delta_j)}
+x_n\eta_n+\mathcal{O}(u^2).
\end{equation*}
 In contrast, the expressions in \cite[Theorem 8]{goldberg1991every}, \cite[page 9-5]{Kahan73},
and \cite[Section 4.2.2]{Knuth2} are equal to
\begin{equation*}
\hat{s}_n=\sum_j{x_j(1+\phi_j)}+\mathcal{O}(nu^2)\sum_k{|x_k|}\qquad \text{where}\quad
|\phi_j|\leq 2u
\end{equation*}
implying the bound $|e_n|\leq 2u \sum_k{|x_k|}+\mathcal{O}(u^2)$.

The reason for the discrepancy with Corollary~\ref{c_cs5}
may be the focus of  \cite[Section 4.3, Appendix]{goldberg1991every} on the
coefficients of $x_1$, which is subjected to one less error than the
other $x_k$.
\end{remark}

\subsection{Second order deterministic and probabilistic bounds}\label{s_csecond}
For  the second order in Algorithm~\ref{alg:compensated}, we present an 
explicit expression and deterministic bound (Theorem~\ref{t_cs6}) and a 
probabilistic bound (Theorem~\ref{t_cs7}).

Theorem~\ref{t_cs6} specifies the second order coefficient in the existing bounds
\cite[Theorem 8]{goldberg1991every}, \cite[(4.9)]{higham2002accuracy},
\cite[Section 4.2.2]{Knuth2} and improves the coefficient $n^2u^2$ in 
\cite[page 9-5]{Kahan73} and \cite[Section4]{Neumaier}.

\begin{theorem}[Explicit expression and deterministic bound]\label{t_cs6}
With assumptions (\ref{e_model1}), define
$\mu_k\equiv \eta_k-\delta_k$, $2\leq k\leq n-1$,
and $\mu_n\equiv \eta_n$. Then 
Algorithm \ref{alg:compensated} has the second order error
\begin{align*}
    e_n = s_n\sigma_n +& (1+\sigma_n)\sum_{k=2}^n{x_k\mu_k}\\
        & -\sum_{k=2}^{n-1}{s_k\sigma_k(\mu_{k+1}+\delta_k+\beta_k)}\\
        &- \sum_{k=2}^{n-1}{x_k\delta_k(\mu_{k+1}+\beta_k+\eta_k)}
        + \mathcal{O}(u^3).
\end{align*}
with the bound
\begin{align*}
|\hat{s}_n-s_n|\leq (3u+4nu^2)\,\sum_{k=1}^n{|x_k|}+\mathcal{O}(u^3).
\end{align*}
\end{theorem}

\begin{proof}
For the second order error in Theorem~\ref{thm:compExpr1},
it suffices to compute the products $P_k\cdots P_{j+1}$ to first order only.
The matrices in Lemma~\ref{l_cs1} equal to first order 
\begin{equation*}
P_k=\begin{bmatrix} 1+\sigma_k & -(1+\sigma_k +\eta_k)\\
\sigma_k& -(\sigma_k+\delta_k)\end{bmatrix} +\mathcal{O}(u^2).
\end{equation*}
Induction shows that $P_k\cdots P_{j+1} =Q_{k:j+1}+ \mathcal{O}(u^2)$
for $k>j+1$, where
\begin{equation}
    Q_{k:j+1}\equiv\begin{bmatrix}
    1 + \sigma_k & \mu_{j+1}-(1+\sigma_k)\\
    \sigma_k & -\sigma_k
    \end{bmatrix}, \qquad k>j+1.
\end{equation}
The matrix in Theorem~\ref{thm:compExpr1} is to second order
\begin{align*}
\widetilde{P}_k 
    = \begin{bmatrix}
    \sigma_k & \eta_k(1+\sigma_k)\\
    \sigma_k(1+\delta_k+\beta_k) & (1+\beta_k)\delta_k + \eta_k(\delta_k+\sigma_k)
    \end{bmatrix} + \mathcal{O}(u^3), \qquad 2\leq k\leq n.
\end{align*}
This implies for the matrix products with $n>j+1$
in Theorem~\ref{thm:compExpr1} 
\begin{align*}
    Q_{n:j+1}\widetilde{P}_j = 
    \begin{bmatrix}
    -\sigma_j(\mu_{j+1}+\delta_j+\beta_j) &  (1+\sigma_n)\mu_j
    -\delta_j(\mu_{j+1}+\beta_j+\eta_j)\\
    0 & \sigma_n\mu_j
    \end{bmatrix} + \mathcal{O}(u^3).
\end{align*}
For $j=n-1$ we have 
\begin{align*}
P_{n}\widetilde{P}_{n-1}=
\begin{bmatrix}-\sigma_{n-1}(\eta_n+\delta_{n-1}+\beta_{n-1})&
(1+\sigma_n)\mu_{n-1}-
\delta_{n-1}(\eta_n+\beta_{n-1}+\eta_{n-1})\\
-\sigma_{n-1}\delta_n& \sigma_n\mu_{n-1}-\delta_n\delta_{n-1}
\end{bmatrix}
\end{align*}
Substitute the above products into the recurrence for the second order error
\begin{align*}
    \begin{bmatrix}
    e_n \\ \hat{c}_n
    \end{bmatrix} = \widetilde{P}_n\begin{bmatrix}
    s_n \\ x_n\end{bmatrix} 
    +P_{n}\widetilde{P}_{n-1}\begin{bmatrix}
    s_{n-1} \\ x_{n-1}\end{bmatrix}
    + \sum_{j=2}^{n-2}Q_{n:j+1}\widetilde{P}_j\begin{bmatrix}
    s_j \\ x_j
    \end{bmatrix} + \mathcal{O}(u^3).
\end{align*}
The error is bounded up to second order by 
\begin{align*}
|e_n|&\leq (3u+6u^2)\sum_{k=1}^n{|x_k|}+4u^2\sum_{k=2}^{n-1}{|s_k|}
\leq (3u+6u^2+4(n-2)u^2)\sum_{k=1}^n{|x_k|}\\
&\leq (3u+4nu^2)\,\sum_{k=1}^n{|x_k|} *\mathcal{O}(u^3).
\end{align*}
\end{proof}

The probabilistic bound below is derived from the explicit expression
in Theorem~\ref{t_cs6}.

\begin{theorem}[Probabilistic second order bound]\label{t_cs7}
Let  $\delta_j$ and  $\epsilon_j$   be independent zero-mean random variables, 
then for any $0<\delta<1$ with probability at least $\delta$, the error in
Algorithm~\ref{alg:shiftSum} is bounded by
\begin{eqnarray*}
|\hat{s}_n-s_n|&\leq& u \left(2(1+3u)\sqrt{\sum_{k=1}^n{|x_k|^2}}
+\sqrt{|s_n|^2+16u^2\sum_{k=1}^{n-1}{|s_k|^2}}\right)\sqrt{2\ln(2/\delta)}+\mathcal{O}(u^3)\\
&\leq &u\left(2(1+3u)\|\mathbf{x}\|_2+\sqrt{1+16(n-2)u^2}\|\mathbf{x}\|_1\right)
\sqrt{2\ln(2/\delta)}+\mathcal{O}(u^3),
\end{eqnarray*}
where $\mathbf{x}=\begin{bmatrix} x_1 & \cdots & x_n\end{bmatrix}^T$.
\end{theorem}

\begin{proof}
Analogously to the proof of Theorem~\ref{t_sprob},
we construct a martingale by unravelling the second order error in Theorem~\ref{t_cs6} in a last-in first-out manner. Let  
$\mu_k=\eta_k-\delta_k$, $2\leq k\leq n-1$, $\mu_n=\eta_n$, $Z_0=0$, and
\begin{align*}
Z_1 &= s_n\sigma_n +x_n\mu_n(1+\sigma_n)\\
Z_k &= s_n\sigma_n  + x_n\mu_n(1+\sigma_n)
-\sum_{j=n-k+1}^{n-1}{s_j\sigma_j(\mu_{j+1}+\delta_j+\beta_j)}\\
&\qquad+ \sum_{j=n-k+1}^{n-1}{x_j
        \left(\mu_j(1+\sigma_n)-\delta_j(\mu_{j+1}+\beta_j+\eta_j)\right)},
        \qquad 2\leq k\leq n-1.
 \end{align*}       
Setting $\delta_n=\beta_n=\eta_2=0$ show that $Z_1$ is a function of 
$\xi_n\equiv(\sigma_n, \eta_n,\delta_n,\beta_n)$ while  $Z_k$ is a function of 
$\xi_j\equiv(\sigma_j,\eta_j,\delta_j, \beta_j)$ 
for $n-k+1\leq j\leq n-1$.

By assumption $\eta_j\equiv(\sigma_j,\eta_j,\delta_j, \beta_j)$ are tuples of 
$4(n-1)$  independent random variables that have zero mean and are bounded,  that is,
\begin{equation*}
\E[\sigma_j]=\E[|\eta_j]=E[\delta_j]=\E[\beta_j]=0,\quad \text{and}\quad
 |\sigma_j|, |\eta_j|, |\delta_j|,|\beta_j|\leq u, \qquad 2\leq j\leq n.
 \end{equation*}
Thus
$\E[Z_{k+1}|\xi_{n-k+1}, \ldots, \xi_n]=Z_k$,  $0\leq k\leq n-2$.
Combine everything to conclude that $Z_0,\ldots, Z_{n-1}$ is a Martingale in 
$\xi_{n},\ldots, \xi_2$.

The differences are  for $1\leq k\leq n-2$
\begin{align*}
Z_1-Z_0 &= s_n\sigma_n + x_n\mu_n(1+\delta_n)\\
Z_{k+1}-Z_k&=-s_{n-k}\sigma_{n-k}(\mu_{n-k+1}+\delta_{n-k}+\beta_{n-k})\\
&\quad +x_{n-k}\left(\mu_{n-k}(1+\sigma_n)-\delta_{n-k}(\mu_{n-k+1}+
\beta_{n-k}+\eta_{n-k})\right).
\end{align*}
Since $|\mu_k|\leq 2u$, $2\leq k\leq n-1$, the differences are bounded by
\begin{eqnarray*}
|Z_1-Z_0|&\leq & c_1\equiv u(v_1+w_1),\qquad
v_1\equiv |s_n|, \qquad w_1\equiv |x_n|(1+u)\\
|Z_{k+1}-Z_k|&\leq & c_{k+1}\equiv u(v_{k+1}+w_{k+1}), \qquad 1\leq k\leq n-2,
\end{eqnarray*}
where
$v_{k+1}\equiv 4 |s_{n-k}| u$ and $w_{k+1}\equiv 2 |x_{n-k}|(1+3u)$.

Lemma~\ref{l_azuma} implies that for any $0<\delta<1$ with probability at least $\delta$
\begin{equation*}
|\hat{s}_n-s_n|=|Z_{n-1}-Z_0|\leq \sqrt{\sum_{k=1}^{n-1}{c_k^2}}\sqrt{2\ln(2/\delta)}.
\end{equation*}
The triangle inequality implies the first  desired inequality 
\begin{equation*}
\sqrt{\sum_{k=1}^{n-1}{c_k^2}}\leq\sqrt{\sum_{k=1}^{n-1}{v_k^2}}+
\sqrt{\sum_{k=1}^{n-1}{w_k^2}},
\end{equation*}
while the second one follows from $|s_k|\leq \|\mathbf{x}\|_1$, $2\leq k\leq n$.
\end{proof}

The above immediately implies a first order bound. 

\begin{corollary}\label{c_cs8}
Let  $\delta_j$ and  $\epsilon_j$   be independent zero-mean random variables, 
then for any $0<\delta<1$ with probability at least $\delta$, the error in
Algorithm~\ref{alg:shiftSum} is bounded by
\begin{eqnarray*}
|\hat{s}_n-s_n|&\leq & u\left(2\|\mathbf{x}\|_2+|s_n|\right)
\sqrt{2\ln(2/\delta)}+\mathcal{O}(u^2),
\end{eqnarray*}
where $\mathbf{x}=\begin{bmatrix} x_1 & \cdots & x_n\end{bmatrix}^T$.
\end{corollary}

%% file: Experiments.tex
\section{Numerical experiments}\label{s_numex}
After describing the setup, 
we present numerical experiments for shifted 
summation (Section~\ref{s_ne1}), compensated summation (Section~\ref{s_ne3}),
and a comparison of the two (Section~\ref{s_ne2}).

The unit roundoffs in the experiments are \cite{fp16}
\begin{itemize}
\item Double precision $u=2^{-53}\approx 1.11\cdot 10^{-16}$.\\
Here double precision is Julia Float64, while `exact' computations are performed
in Julia Float256.
\item Half precision $u=2^{-11}\approx 4.88 \cdot 10^{-4}$.
Here half precision is Julia Float16, while `exact' computations are performed
in Julia Float64.
\end{itemize}

We plot relative errors $|\hat{s}_n-s_n|/|s_n|$ rather than absolute errors
to allow for meaningful calibration:
Relative errors $\leq u$ indicate full accuracy; while
relative errors $\geq .5$ indicate zero accuracy.

For shifted summation we use the empirical mean of two extreme summands,
\begin{equation*}
c=(\min_k{x_k}+\max_k{x_k})/2.
\end{equation*}
For probabilistic bounds, the failure probability is $\delta=10^{-2}$, hence
$\sqrt{2\ln(2/\delta)}\approx 3.26$.

Two classes of summands will be considered:
\begin{itemize}
\item $x_k=m +y_k$ where $y_k$ are independent uniform$[0,1]$ random variables, 
with $0\leq y\leq 1$; and $m=10^4$
in Sections \ref{s_ne1} and~\ref{s_ne2} and  $m=0$ in Section~\ref{s_ne3}.\\
Since all summands have the same sign, this is a well-conditioned
problem with $|s_n|=\sum_{k=1}^n{|x_k|}$. 
The case $m=10^4$ is designed to increase the accuracy in shifted summation
since $x_k$ are tightly clustered at $10^4$. 
\item $x_k$ are normal$(0,1)$ random variables with mean 0 and variance 
Summands have different signs, which may lead to cancellation. The $x_k$
are naturally clustered around zero.
\end{itemize}

The plots show relative errors $|\hat{s}_n-s_n|/|s_n|$ versus $n$,
where $10\leq n\leq n_{max}$.
To speed up the computations, we compute  $x_1,\ldots, x_{n_{max}}$ once,
and then pick its leading $n$ elements as $x_1, \ldots, x_n$.

\subsection{Errors and bounds for sequential summation}\label{s_ne1}
We compare the errors from sequential versions of Algorithm \ref{alg:sum}
and~\ref{alg:shiftSum}, that is, sequential summation without and with shifting, and
an easy relative version of the probabilistic bound in Theorem~\ref{t_sprob}
\begin{equation}\label{e_ne1}
u \sqrt{n+2}\,\max_{1\leq k\leq n+1}{\frac{|s_k-kc|+|x_k-c|}{|s_n|}}\,
\sqrt{2\ln{(2/\delta)}}
\end{equation}

\begin{figure}
\begin{center}
\includegraphics[width = 2.5in]{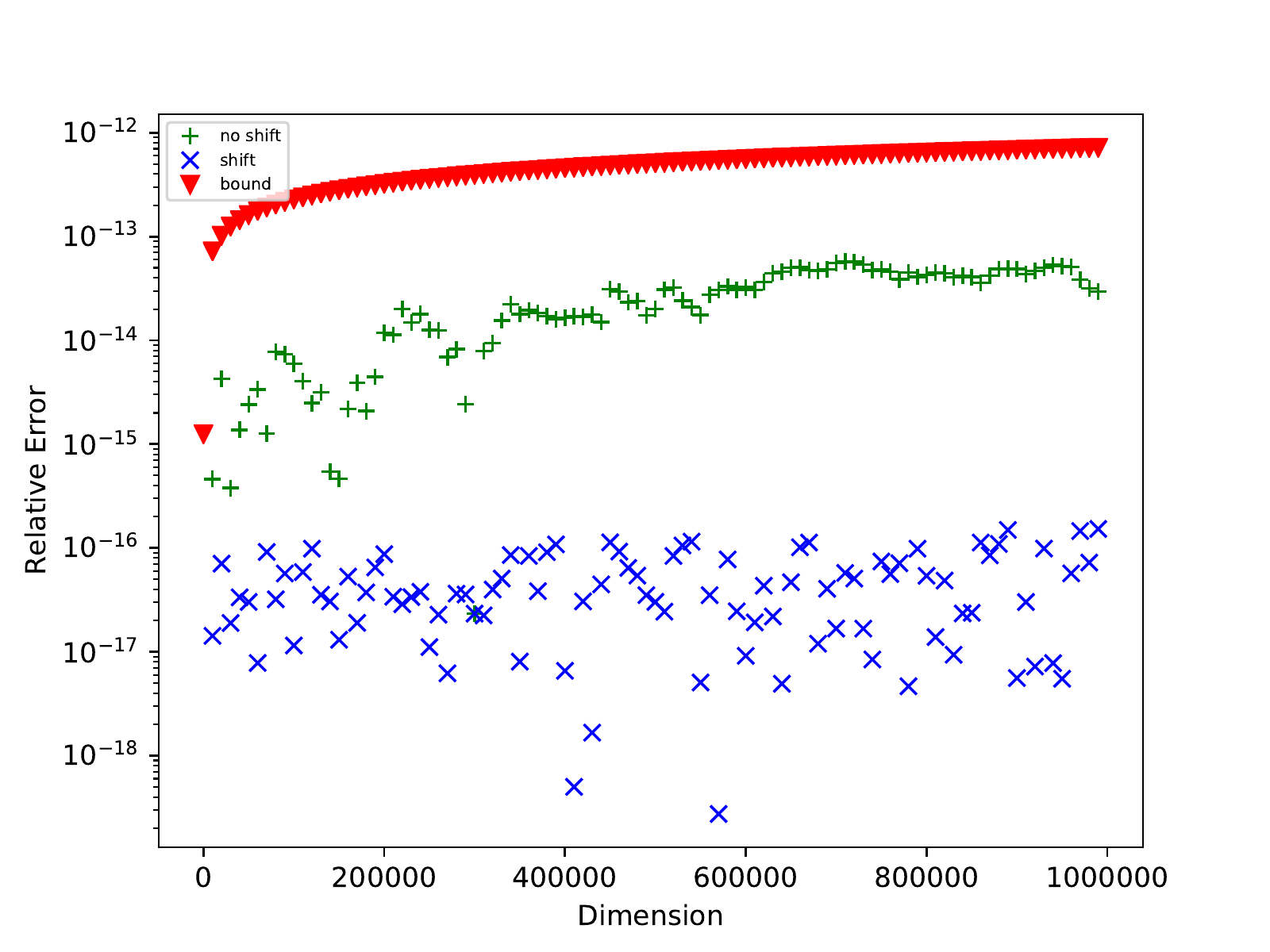}
\includegraphics[width = 2.5in]{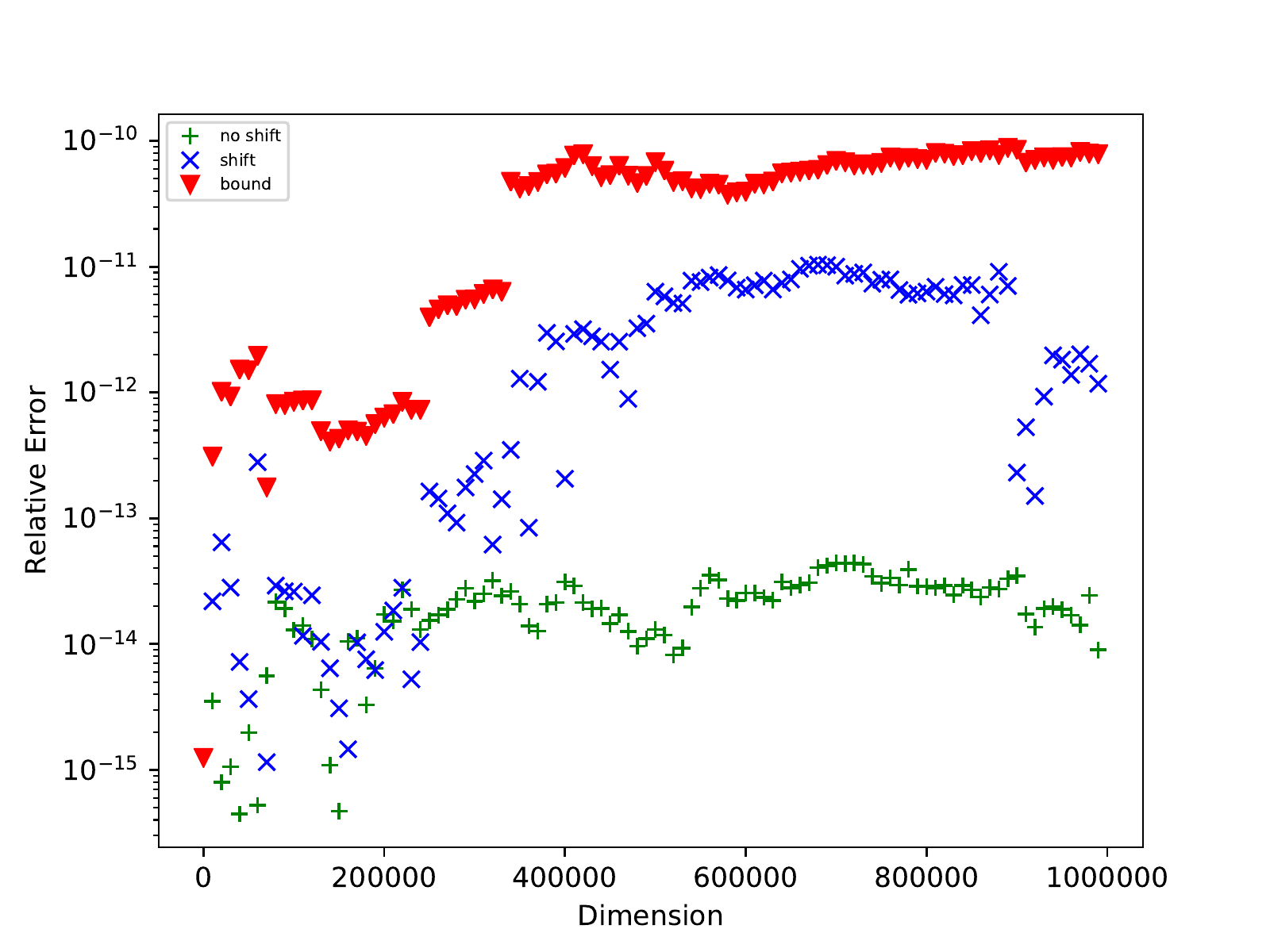}
\end{center}
\caption{Relative errors in plain
summation (green +) and summation with shift (blue x); and~(\ref{e_ne1})
(red triangle).
Left panel: $x_k=10^4+\mathrm{uniform}[0,1]$. Right panel: $x_k=\mathrm{normal}(0,1)$.}
\label{f_ne1}
\end{figure}

In Figure~\ref{f_ne1}, the range for the number of summands is $n=10:10^4:10^6$, 
and working precision is double with $u\approx 10^{-16}$.
Figure~\ref{f_ne1} illustrates that, compared to plain summation,
shifted summation can increase the accuracy
when summands are artificially and tightly clustered at a large mean, however
it can hurt the accuracy when summands are naturally clustered around zero.
For the normal(0,1) summands in the right panel, the range of shifts 
is $.003\leq c\leq  1.2$. The expression (\ref{e_ne1}) represents an upper bound
accurate within a factor 10-100.

\subsection{Comparison of shifted and compensated summation}\label{s_ne2}
We compare the errors from sequential summation with shifting 
(sequential version of Algorithm \ref{alg:sum}) and compensated summation
(Algorithm~\ref{alg:compensated}).

\begin{figure}
\begin{center}
\includegraphics[width = 2.5in]{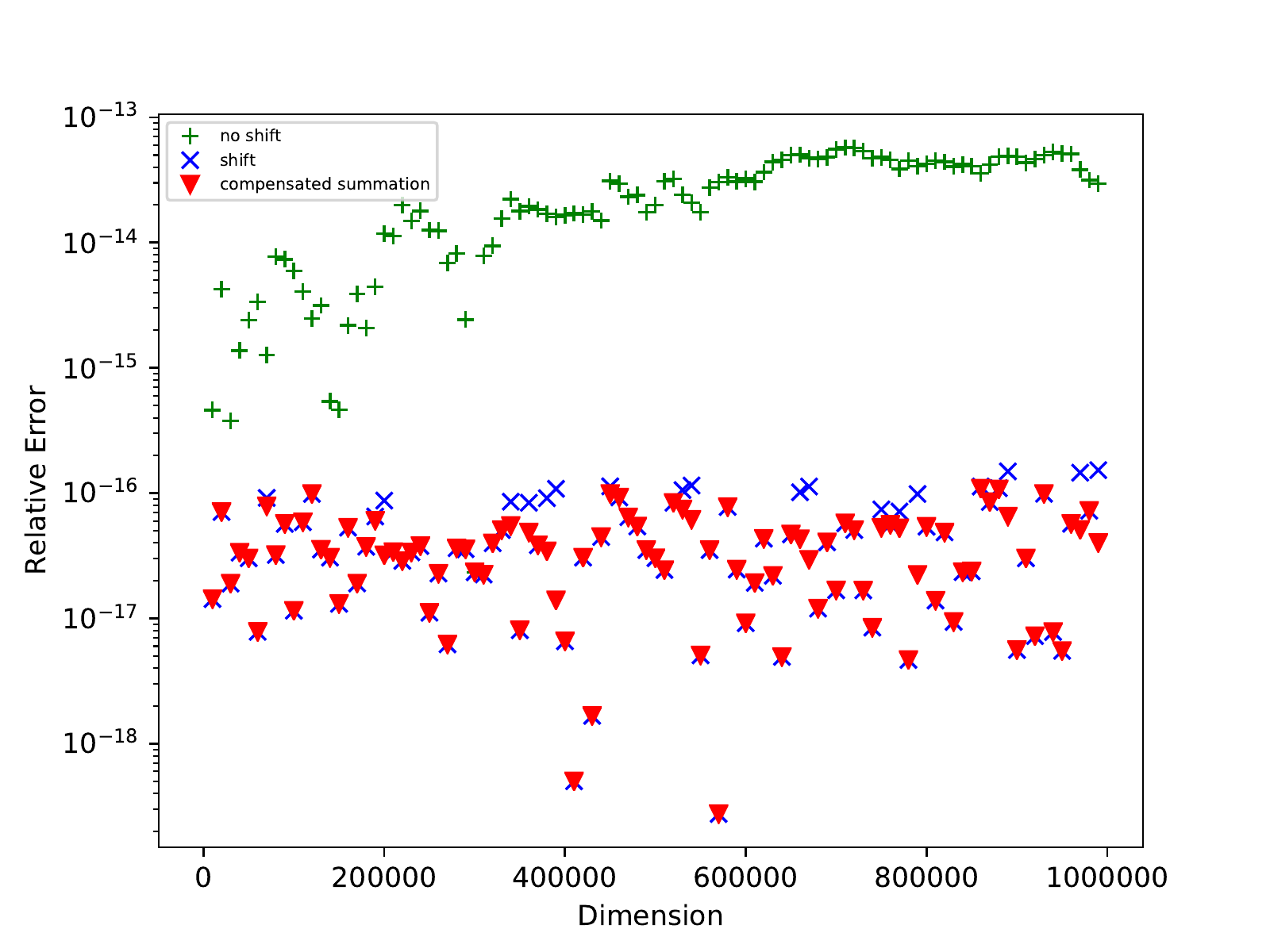}
\includegraphics[width = 2.5in]{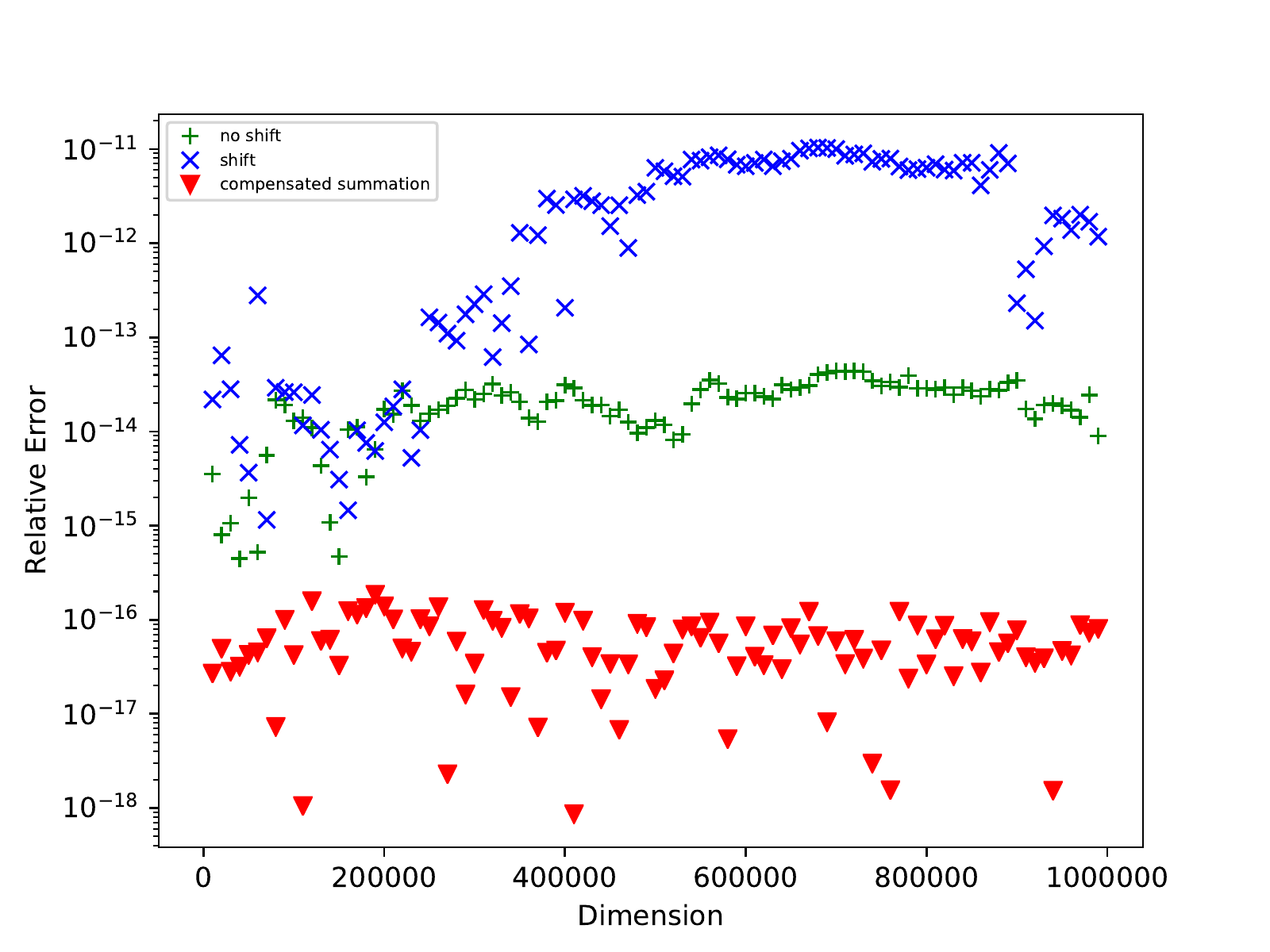}
\end{center}
\caption{Relative errors in plain
summation (green +), summation with shift (blue x), and compensated summation
(red triangle).
Left panel: $x_k=10^4+\mathrm{uniform}[0,1]$. Right panel: $x_k=\mathrm{normal}(0,1)$.}
\label{f_ne2}
\end{figure}

In Figure~\ref{f_ne2}, the range for the number of summands is $n=10:10^4:10^6$, 
and working precision is double with $u\approx 10^{-16}$.
Figure~\ref{f_ne2} illustrates that compensated summation is at least as
accurate and can be much more accurate than shifted summation. in both panels,
compensated summation is accurate to machine precision.

\subsection{Errors and bounds for compensated summation}\label{s_ne3}
We compare the errors from sequential summation (sequential version of Algorithm \ref{alg:sum})
and compensated summation (Algorithm~\ref{alg:compensated}) with the relative version of
the first order probabilistic bound in Corollary~\ref{c_cs8}
\begin{equation}\label{e_ne3}
u\,\frac{2\|\mathbf{x}\|_2+|s_n|}{|s_n|}\sqrt{2\ln(2/\delta)},\qquad
\mathbf{x}=\begin{bmatrix} x_1 & \cdots & x_n\end{bmatrix}^T.
\end{equation}

\begin{figure}
\begin{center}
\includegraphics[width = 2.5in]{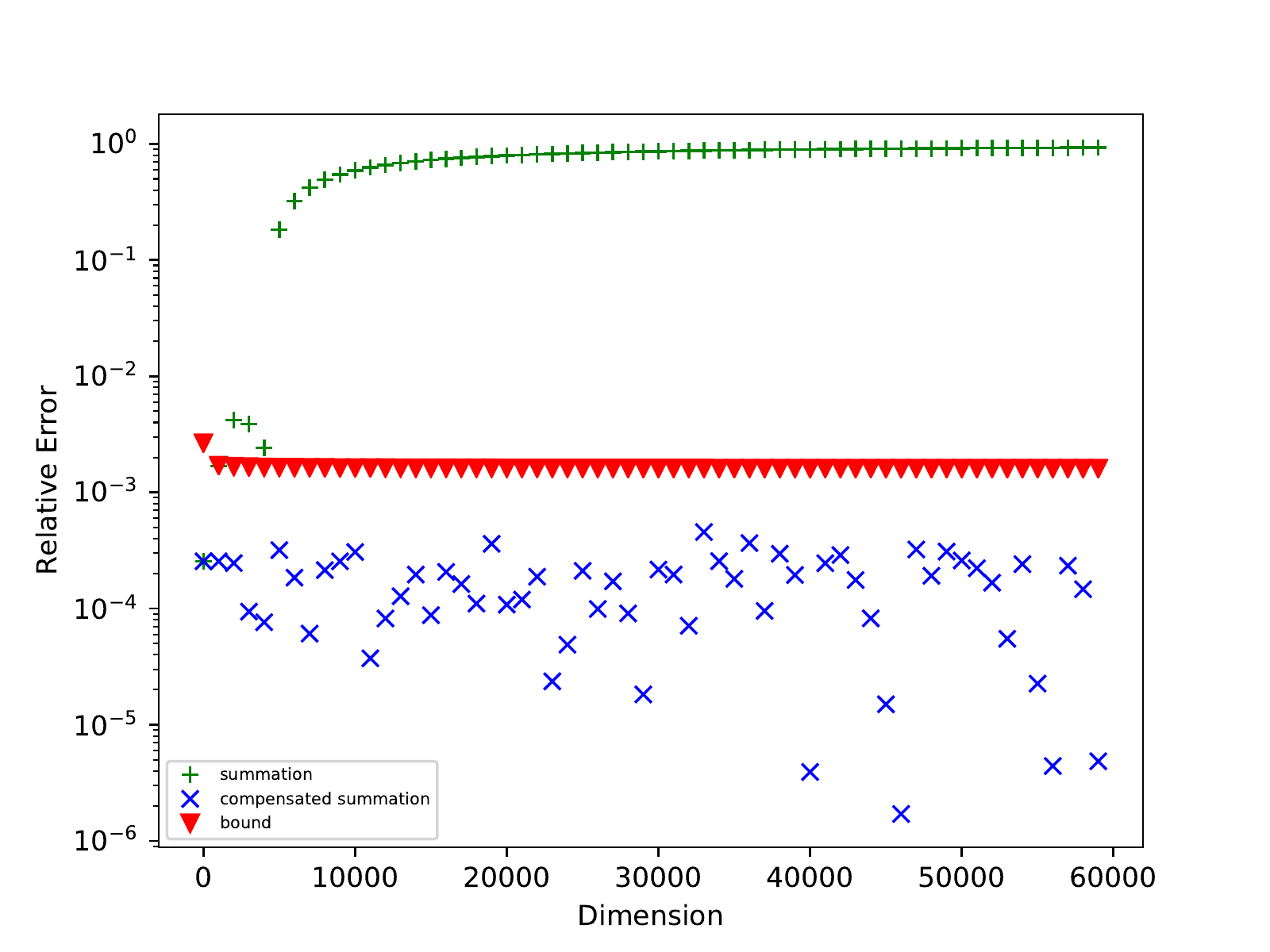}
\includegraphics[width = 2.5in]{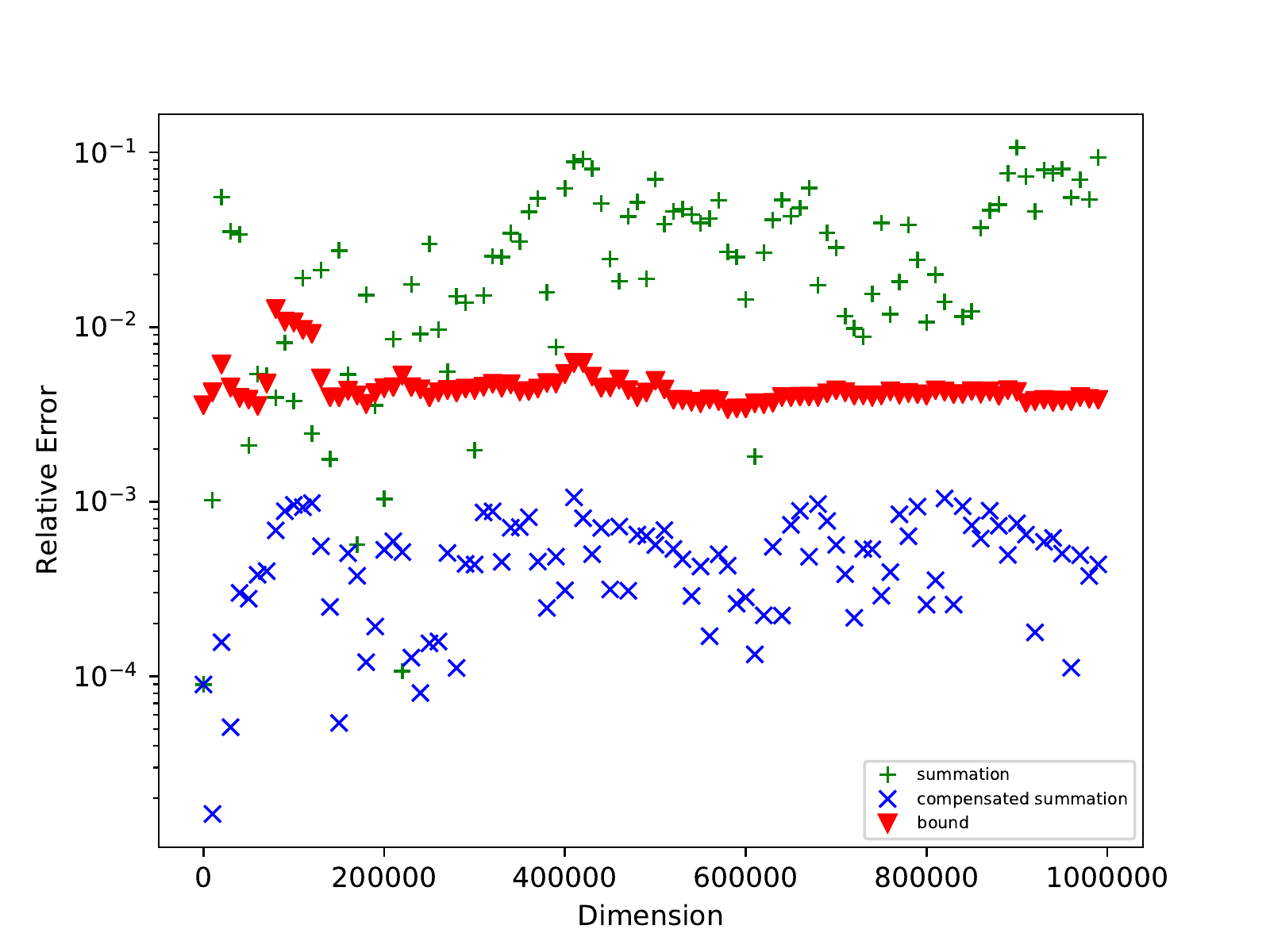}
\end{center}
\caption{Relative errors in plain
summation (green +), compensated (blue x), and~(\ref{e_ne3}) (red triangle).
Left panel: $x_k=\mathrm{uniform}[0,1]$.
Right panel: $x_k=\mathrm{normal}(0,1)$.}
\label{f_ne3}
\end{figure}

In Figure~\ref{f_ne3}, the working precision is double with $u\approx 5\cdot 10^{-4}$.
In the left panel the number of uniform(0,1) summands is $n=10:10^3:6\cdot 10^4$, 
to avoid overflow since the
largest value in binary16 is about 65,504. In the right panel
in contrast, the number of normal(0,1) summands is $n=10:10^4:10^6$, as summands
of different signs are less likely to cause overflow.

Figure~\ref{f_ne3} illustrates that, in contrast to plain summation,
compensated summation is accurate to machine precision, even for $n>u^{-1}$
in the right panel, and that (\ref{e_ne3}) represents an upper bound
accurate with a factor 10.